\documentclass[11pt,a4paper]{article}
\usepackage[utf8]{inputenc}
\usepackage[T1]{fontenc}
\usepackage[french,english]{babel}
\usepackage{amsmath, amssymb, amsthm, amsfonts}
\usepackage{graphicx}
\usepackage{booktabs}
\usepackage{hyperref}
\usepackage[margin=2.5cm]{geometry}
\usepackage{array}
\usepackage{multirow}

\newtheorem{theorem}{Theorem}[section]

\newtheorem{proposition}[theorem]{Proposition}
\newtheorem{definition}[theorem]{Definition}
\newtheorem{remark}[theorem]{Remark}

\newcommand{\E}{\mathbb{E}}

\newcommand{\hatp}{\hat{\mathbf{p}}_n}
\DeclareMathOperator{\Var}{Var}

\DeclareMathOperator{\Bias}{Bias}

\title{Finite-Sample Bias Correction for Plug-in Estimators of Extropy, Rényi Extropy, and Tsallis Extropy}
\author{Amadou Diadie BA\\
LERSTAD, Gaston Berger University, Saint-Louis, Sénégal\\
ba1amadou@yahoo.fr\ \ \ \ ba.amadou-diadie@ugb.edu.sn}
\date{}

\begin{document}

\maketitle

\begin{abstract}
This paper studies the finite-sample bias of plug-in estimators for Shannon, Rényi, and Tsallis extropies for finitely supported discrete random variables. We derive explicit first-order bias formulas for all three estimators by analyzing the bias of the intermediate functional \(S_\alpha = \sum_{i=1}^r (1-p_i)^\alpha\). Bias-corrected estimators are then proposed. We show that the uncorrected and corrected estimators are asymptotically equivalent, differing only by terms of order \(O(1/n)\). A simulation study validates the theoretical results and demonstrates the superior finite-sample performance of the bias-corrected estimators. The results justify the use of the simpler uncorrected plug-in estimators in asymptotic settings, such as those considered in the companion paper \cite{companion} on almost sure convergence and asymptotic normality.
\end{abstract}

\noindent \textbf{Mathematics Subject Classifications (MSC2020):} 62G05, 62G20, 62B10, 94A17\\

\section{Introduction}

\subsection{Motivation}

While entropy has long been the dominant measure of uncertainty in probability and information theory, extropy offers a complementary perspective by quantifying the uncertainty associated with the non-occurrence of events. Extropy is particularly sensitive to distribution tails and rare scenarios, making it a valuable tool in fields such as risk analysis, decision theory, finance, and reliability engineering.

The plug-in estimators of extropy and its generalizations are intuitive and easy to compute. They are also consistent. Despite that, they suffer from bias in finite samples. The presence of bias can significantly affect the accuracy of extropy-type estimates when the sample size \(n\) is moderate. Bias correction is therefore important for improving finite-sample accuracy, especially when the parameter \(\alpha\) departs from 1 or when the support size \(r\) is moderate.

\subsection{Applications}

Extropy has emerged as a fundamental dual measure to entropy with significant implications across disciplines. In risk analysis and finance, extropy highlights rare catastrophic outcomes \cite{shi2025, bruno2020}. In machine learning, it provides tools for uncertainty quantification in classifier outputs \cite{jawal2022, kharazmi2026}. In reliability and survival analysis, cumulative residual extropy offers refined tools for modeling system lifetimes \cite{diCrescenzo2019, asadi2018}. In pattern recognition and health data analysis, Tsallis extropy has been applied as a measure of discrimination \cite{balakrishnan2022tsallis, aldallal2025}. Extropy also complements entropy-based measures in information theory and statistics \cite{hood2015, qiu2018estimators}.

\subsection{Previous Work and Motivation for Bias Correction}

The plug-in estimators defined in \eqref{eq:J_hat}, \eqref{eq:JR_hat}, and \eqref{eq:JT_hat} are intuitive and easy to compute. They are also consistent. Despite that, they suffer from bias in finite samples:
\[
\Bias(\widehat{J}) = \E[\widehat{J}] - J(\mathbf{p}) \neq 0,
\]
and analogously for Rényi and Tsallis extropies estimators.

The presence of bias can significantly affect the accuracy of extropy-type estimates when the sample size \(n\) is moderate. Bias correction is therefore important for improving finite-sample accuracy, especially when the parameter \(\alpha\) departs from 1 or when the support size \(r\) is moderate.

Various methods have been developed in the literature to reduce bias or otherwise improve the estimation of entropy and extropy-type measures. These include:

\begin{enumerate}
    \item \textbf{Kernel-based methods:} \cite{jahanshahi2020nonparametric} proposed kernel-based density estimators for extropy. While these methods can improve performance for continuous distributions, they require careful bandwidth selection and are less straightforward for discrete data.
    
    \item \textbf{Bayesian approaches:} \cite{allabadi2020bayesian} developed a Bayesian approach based on the Dirichlet process for the estimation of extropy. This method is computationally intensive and may be sensitive to prior specification.
    
    \item \textbf{Bootstrap-based methods:} Bootstrap resampling techniques can be used to estimate and correct bias, but they are computationally expensive.
    
    \item \textbf{Analytical bias correction:} This approach, which we adopt in this paper, consists of deriving explicit expressions for the leading bias term using Taylor expansions and then subtracting an estimate of the bias from the plug-in estimator.
\end{enumerate}

We choose the \textbf{analytical bias correction} approach for the following reasons:

\begin{enumerate}
    \item \textbf{Computational efficiency:} No additional simulations are required.
    \item \textbf{Interpretability:} The bias formulas reveal how the bias depends on \(p_i\), \(n\), and \(\alpha\).
    \item \textbf{Explicit asymptotic equivalence:} We prove that the uncorrected and corrected versions differ only by \(O(1/n)\).
    \item \textbf{Simplicity:} No tuning parameters are needed.
\end{enumerate}

In this paper, we derive explicit expressions for the bias of the Shannon, Rényi, and Tsallis extropy estimators at order \(O(1/n)\). These expressions are then used to construct bias-corrected estimators that significantly improve finite-sample accuracy while preserving asymptotic consistency.

\begin{remark}
The analytical bias correction approach is particularly well-suited for this problem because the extropy functionals are smooth functions of the multinomial probabilities. This smoothness allows us to use Taylor expansions to obtain accurate bias approximations.
\end{remark}

\subsection{Main Contribution}

Our main contribution may be summarized as follows. Given an i.i.d. sample \(X_1,\dots,X_n\) from \(X\) with distribution \(\mathbf{p}\), we derive explicit first-order bias formulas and propose bias-corrected estimators for Shannon, \(\alpha\)-Rényi, and \(\alpha\)-Tsallis extropies.

Our method consists of first obtaining the bias of the intermediate functional
\[
S_\alpha(\mathbf{p}) = \sum_{i=1}^r (1-p_i)^\alpha,
\]
using a second-order Taylor expansion. This result is then used to derive the bias of each extropy estimator.

The main results are:
\begin{enumerate}
    \item Explicit bias formulas for Shannon, Rényi, and Tsallis extropy estimators at order \(O(1/n)\);
    \item Bias-corrected estimators that significantly improve finite-sample accuracy;
    \item Proof that the uncorrected and corrected estimators are asymptotically equivalent, differing only by \(O(1/n)\);
    \item A simulation study validating the theoretical results and demonstrating the practical relevance of the bias-corrected estimators.
\end{enumerate}

This work provides a theoretical justification for using the simpler uncorrected plug-in estimators in asymptotic settings, such as those considered in the companion paper \cite{companion}, where almost sure convergence and asymptotic normality are established.

\subsection{Overview of the Paper}

The remainder of the paper is organized as follows. In Section \ref{section2}, we present the notations and preliminary results. Section \ref{section3} derives the bias-corrected estimators for Shannon, Rényi, and Tsallis extropies. Section \ref{section4} presents a comprehensive simulation study. Finally, Section \ref{section5} concludes the paper and discusses future research directions.

\section{Notations and Preliminary Results}
\label{section2}
\subsection{The Discrete Random Variable and Its Estimator}

\begin{definition}
Let \(X\) be a discrete random variable defined on the probability space \((\Omega, \mathcal{A}, \mathbb{P})\) with finite support \(\mathcal{X} = \{c_1, \dots, c_r\}\) (\(r \ge 2\)) and probability mass function \(p_i = \mathbb{P}(X = c_i)\) for \(i \in I = \{1, \dots, r\}\).
\end{definition}

In general, the full probability distribution \(\mathbf{p} = (p_i)_{1 \le i \le r}\) is unknown. Let \(X_1, \dots, X_n\) be i.i.d. according to \(\mathbf{p}\). For \(i \in I\), define the empirical estimator
\begin{equation}
\widehat{p}_n^i = \frac{1}{n}\sum_{j=1}^n \mathbf{1}_{c_i}(X_j), \label{eq:emp_estimator}
\end{equation}
where
$\displaystyle
\mathbf{1}_{c_i}(X_j) = 
\begin{cases}
1 & \text{if } X_j = c_i,\\
0 & \text{otherwise}.
\end{cases}
$
\noindent For fixed \(i\in I\), \(n\widehat{p}_n^i \sim \text{Binomial}(n, p_i)\), hence
\begin{equation}
\E[\widehat{p}_n^i] = p_i, \quad \Var(\widehat{p}_n^i) = \frac{p_i(1-p_i)}{n}. \label{eq:emp_moments}
\end{equation}

\subsection{The Functional \(S_\alpha\) and Its Bias}

In the study of Rényi-Tsallis extropies, a central quantity is the functional
\begin{equation}
S_\alpha(\mathbf{p}) = \sum_{i=1}^r (1-p_i)^\alpha, \qquad \alpha \in (0,1)\cup(1,\infty), \label{eq:S_alpha}
\end{equation}
and its empirical version
\begin{equation}
\widehat{S}_\alpha = \sum_{i=1}^r (1-\widehat{p}_n^i)^\alpha. \label{eq:S_hat}
\end{equation}

The quantity \(S_\alpha\) appears naturally as an intermediate functional in the definition of several generalized extropies, and its behavior governs the corresponding statistical properties of their plug-in estimators.

The following proposition establishes the bias of \(\widehat{S}_\alpha\) at order \(1/n\).

\begin{proposition}[Bias of \(\widehat{S}_\alpha\)]
Let \(\widehat{p}_n\) be the empirical frequency vector from a multinomial sample of size \(n\). Then
\begin{equation}
\E[\widehat{S}_\alpha] = S_\alpha(\mathbf{p}) + \frac{1}{n} B_{S_\alpha}(\mathbf{p},\alpha) + o\left(\frac{1}{n}\right), \label{eq:S_bias}
\end{equation}
with
\begin{equation}
B_{S_\alpha}(\mathbf{p},\alpha) = \frac{\alpha(\alpha-1)}{2}\sum_{i=1}^r p_i(1-p_i)^{\alpha-1}. \label{eq:B_Salpha}
\end{equation}
\end{proposition}

\begin{proof}
We keep the notations:
\[
\widehat{S}_\alpha = \sum_{i=1}^r \phi(\widehat{p}_n^i), \qquad \phi(t) = (1-t)^\alpha, \ \ t<1 \qquad S_\alpha = \sum_{i=1}^r \phi(p_i).
\]

By Taylor expansion of each \(\phi(\widehat{p}_n^i)\) at \(p_i\) up to second order:
\[
S_\alpha(\hatp) = S_\alpha(\mathbf{p}) + \sum_{i=1}^r \phi'(p_i)\Delta_n^i + \frac{1}{2}\sum_{i=1}^r \phi''(p_i)(\Delta_n^i)^2 + o_p(\|\Delta\|^2),
\]
where \(\Delta_n^i = \widehat{p}_n^i - p_i\), \(\phi'(t) = -\alpha(1-t)^{\alpha-1}\), and \(\phi''(t) = \alpha(\alpha-1)(1-t)^{\alpha-2}\).

Since \(\E[\Delta_n^i] = 0\) and \(\Var(\widehat{p}_i) = p_i(1-p_i)/n\), we obtain:
\[
\E[S_\alpha(\hatp)] - S_\alpha(\mathbf{p}) = \frac{\alpha(\alpha-1)}{2n}\sum_{i=1}^r p_i(1-p_i)^{\alpha-1} + o\left(\frac{1}{n}\right).
\]

Thus,
\[
\E[\widehat{S}_\alpha] = S_\alpha(\mathbf{p}) + \frac{B_{S_\alpha}(\mathbf{p},\alpha)}{n} + o\left(\frac{1}{n}\right).
\]

This proves \eqref{eq:S_bias} and \eqref{eq:B_Salpha}.
\end{proof}

\begin{remark}
The bias of \(\widehat{S}_\alpha\) is of order \(O(1/n)\) and vanishes as \(n \to \infty\). The sign of the bias depends on \(\alpha\): if \(\,0<\alpha < 1\), the bias is negative; if \(\alpha > 1\), the bias is positive.
\end{remark}

\section{Extropies and Bias-Corrected Estimators}
\label{section3}
\subsection{Definitions of Extropies}

\begin{definition}[Shannon Extropy]
The Shannon extropy of the random variable \(X\) is given by (see \cite{lad2015extropy})
\begin{equation}
J(\mathbf{p}) = -\sum_{i=1}^r (1-p_i)\log(1-p_i). \label{eq:shannon}
\end{equation}
For ease of computations, we use the natural logarithm.
\end{definition}

Two important one-parameter families of extropy measures have been proposed:

\begin{definition}[Rényi Extropy]
The Rényi extropy of the random variable \(X\) is defined by (see \cite{liu2023renyi})
\begin{equation}
J_{R,\alpha}(\mathbf{p}) = \frac{1}{1-\alpha}\log\left(\sum_{i=1}^r (1-p_i)^\alpha\right), \quad \alpha \in (0,1)\cup(1,\infty). \label{eq:renyi}
\end{equation}
\end{definition}

\begin{definition}[Tsallis Extropy]
The Tsallis extropy of the random variable \(X\) is defined by (see \cite{balakrishnan2022tsallis})
\begin{equation}
J_{T,\alpha}(\mathbf{p}) = \frac{1}{1-\alpha}\left(\sum_{i=1}^r (1-p_i)^\alpha - 1\right), \quad \alpha \in (0,1)\cup(1,\infty). \label{eq:tsallis}
\end{equation}
\end{definition}

As \(\alpha \to 1\), both measures converge to the Shannon extropy \(J(\mathbf{p})\). The parameter \(\alpha\) regulates sensitivity to small or large probabilities: values \(\alpha < 1\) give more weight to rare events, while \(\alpha > 1\) emphasize dominant events.

\subsection{Plug-in Estimators}

The corresponding plug-in estimators are obtained by replacing \(\mathbf{p}\) by \(\hatp\):
\begin{equation}
\widehat{J} = -\sum_{i=1}^r (1-\widehat{p}_n^i)\log(1-\widehat{p}_n^i), \label{eq:J_hat}
\end{equation}
\begin{equation}
\widehat{J}_{R,\alpha} = \frac{1}{1-\alpha}\log\left(\sum_{i=1}^r (1-\widehat{p}_n^i)^\alpha\right), \label{eq:JR_hat}
\end{equation}
\begin{equation}
\widehat{J}_{T,\alpha} = \frac{1}{1-\alpha}\left(\sum_{i=1}^r (1-\widehat{p}_n^i)^\alpha - 1\right). \label{eq:JT_hat}
\end{equation}

These plug-in estimators are intuitive and easy to compute. They are also consistent. Despite that, they suffer from bias in finite samples:
\[
\Bias(\widehat{J}) = \E[\widehat{J}] - J(\mathbf{p}) \neq 0,
\]
and analogously for Rényi and Tsallis extropies estimators.

\subsection{Bias Correction}

The following proposition gives the bias-corrected estimators of the plug-in extropies.

\begin{proposition}[Bias-corrected extropy estimators]
Let \(\widehat{J}\), \(\widehat{J}_{R,\alpha}\), and \(\widehat{J}_{T,\alpha}\) be the plug-in estimators defined in \eqref{eq:J_hat}-\eqref{eq:JT_hat}. The bias-corrected estimators are:
\begin{equation}
\widehat{J}^{(\text{corr})} = \widehat{J} + \frac{1}{2n}, \label{eq:J_corr}
\end{equation}
\begin{equation}
\widehat{J}_{T,\alpha}^{(\text{corr})} = \widehat{J}_{T,\alpha} + \frac{\alpha}{2n}\sum_{i=1}^r \widehat{p}_n^i (1-\widehat{p}_n^i)^{\alpha-1}, \label{eq:JT_corr}
\end{equation}
\begin{equation}
\widehat{J}_{R,\alpha}^{(\text{corr})} = \widehat{J}_{R,\alpha} - \frac{1}{2n(1-\alpha)}\left(\frac{\alpha(\alpha-1)\widehat{T}_1}{\widehat{S}_\alpha} - \frac{\alpha^2(\widehat{T}_2 - \widehat{T}_1^2)}{\widehat{S}_\alpha^2}\right), \label{eq:JR_corr}
\end{equation}
where
\begin{equation}
\widehat{T}_1 = \sum_{i=1}^r \widehat{p}_n^i (1-\widehat{p}_n^i)^{\alpha-1}, \quad
\widehat{T}_2 = \sum_{i=1}^r \widehat{p}_n^i (1-\widehat{p}_n^i)^{2\alpha-2}, \quad
\widehat{S}_\alpha = \sum_{i=1}^r (1-\widehat{p}_n^i)^\alpha. \label{eq:T_def}
\end{equation}
\end{proposition}

\begin{proof}
We use the bias expansion of \(\widehat{S}_\alpha\) established in Proposition 2.2.

\subsection*{Bias of the Tsallis extropy estimator}

For the Tsallis extropy:
\[
J_{T,\alpha}(\mathbf{p}) = \frac{1}{1-\alpha}(S_\alpha(\mathbf{p}) - 1).
\]

The bias is:
\[
\Bias(\widehat{J}_{T,\alpha}) = \frac{1}{1-\alpha}\left(\E[\widehat{S}_\alpha] - S_\alpha(\mathbf{p})\right).
\]

Using \eqref{eq:S_bias}:
\[
\Bias(\widehat{J}_{T,\alpha}) = \frac{1}{1-\alpha} \cdot \frac{B_{S_\alpha}(\mathbf{p},\alpha)}{n} + o\left(\frac{1}{n}\right)
= -\frac{\alpha}{2n}\sum_{i=1}^r p_i(1-p_i)^{\alpha-1} + o\left(\frac{1}{n}\right). \label{eq:bias_T}
\]

\begin{remark}
Since \(\alpha > 0\), \(p_i \in (0,1)\), and \(n > 0\), we have \(\frac{\alpha}{2n} > 0\) and \(\sum_{i=1}^r p_i(1-p_i)^{\alpha-1} > 0\). Therefore, \(\Bias(\widehat{J}_{T,\alpha}) < 0\) for all \(\alpha > 0\), \(\alpha \neq 1\). The Tsallis extropy estimator systematically \emph{underestimates} the true value in finite samples.
\end{remark}

A first-order bias-corrected Tsallis estimator is therefore:
\[
\widehat{J}_{T,\alpha}^{(\text{corr})} = \widehat{J}_{T,\alpha} - \widehat{\Bias}(\widehat{J}_{T,\alpha})
= \widehat{J}_{T,\alpha} + \frac{\alpha}{2n}\sum_{i=1}^r \widehat{p}_i(1-\widehat{p}_i)^{\alpha-1}.
\]

\subsection*{Bias of the Rényi extropy estimator}

For the Rényi extropy:
\[
J_{R,\alpha}(\mathbf{p}) = \frac{1}{1-\alpha}\log(S_\alpha(\mathbf{p})).
\]

Let \(Y_n = \widehat{S}_\alpha - S_\alpha\). Then:
\[
\log \widehat{S}_\alpha = \log S_\alpha + \frac{Y_n}{S_\alpha} - \frac{Y_n^2}{2S_\alpha^2} + o_p(Y_n^2).
\]

Taking expectations:
\[
\E[\log \widehat{S}_\alpha] - \log S_\alpha = \frac{\E[Y_n]}{S_\alpha} - \frac{\E[Y_n^2]}{2S_\alpha^2} + o\left(\frac{1}{n}\right).
\]

Now:
\[
\E[Y_n] = \frac{B_{S_\alpha}}{n} + o\left(\frac{1}{n}\right) = \frac{\alpha(\alpha-1)}{2n}T_1 + o\left(\frac{1}{n}\right),
\]
where \(\displaystyle T_1 = \sum_{i=1}^r p_i(1-p_i)^{\alpha-1}\).

And:
\[
\E[Y_n^2] = \Var(\widehat{S}_\alpha) + (\E[Y_n])^2 = \frac{\alpha^2}{n}(T_2 - T_1^2) + o\left(\frac{1}{n}\right),
\]
where \(\displaystyle T_2 = \sum_{i=1}^r p_i(1-p_i)^{2\alpha-2}\).

Hence:
\[
\E[\log \widehat{S}_\alpha] - \log S_\alpha = \frac{\alpha(\alpha-1)T_1}{2nS_\alpha} - \frac{\alpha^2(T_2 - T_1^2)}{2nS_\alpha^2} + o\left(\frac{1}{n}\right).
\]

Therefore:
\[
\Bias(\widehat{J}_{R,\alpha}) = \frac{1}{1-\alpha}\left(\E[\log \widehat{S}_\alpha] - \log S_\alpha\right)
= \frac{1}{2n(1-\alpha)}\left(\frac{\alpha(\alpha-1)T_1}{S_\alpha} - \frac{\alpha^2(T_2 - T_1^2)}{S_\alpha^2}\right) + o\left(\frac{1}{n}\right). \label{eq:bias_R}
\]

A first-order bias-corrected Rényi estimator is therefore:
\[
\widehat{J}_{R,\alpha}^{(\text{corr})} = \widehat{J}_{R,\alpha} - \frac{1}{2n(1-\alpha)}\left(\frac{\alpha(\alpha-1)\widehat{T}_1}{\widehat{S}_\alpha} - \frac{\alpha^2(\widehat{T}_2 - \widehat{T}_1^2)}{\widehat{S}_\alpha^2}\right).
\]

\subsection*{Bias of the Shannon extropy estimator}

The Shannon extropy is obtained as the limit \(\alpha \to 1\):
\[
J(\mathbf{p}) = -\sum_{i=1}^r (1-p_i)\log(1-p_i).
\]

We derive the bias directly by a Taylor expansion of \(\psi(t) = -(1-t)\log(1-t)\). We have:
\[
\psi'(t) = 1 + \log(1-t), \quad \psi''(t) = -\frac{1}{1-t}.
\]

A Taylor expansion gives:
\[
\E[\psi(\widehat{p}_n^i)] = \psi(p_i) + \frac{1}{2}\psi''(p_i)\frac{p_i(1-p_i)}{n} + o\left(\frac{1}{n}\right)
= \psi(p_i) - \frac{p_i}{2n} + o\left(\frac{1}{n}\right).
\]

Summing over \(i\):
\[
\E[\widehat{J}] - J(\mathbf{p}) = -\frac{1}{2n} + o\left(\frac{1}{n}\right). \label{eq:bias_J}
\]

\begin{remark}
The Shannon extropy estimator also systematically \emph{underestimates} the true value in finite samples, with a bias of \(-1/(2n)\). Interestingly, the bias does not depend on the individual probabilities \(p_i\), only on the sample size \(n\).
\end{remark}

A first-order bias-corrected Shannon estimator is therefore:
\[
\widehat{J}^{(\text{corr})} = \widehat{J} + \frac{1}{2n}.
\]

This completes the proof of Proposition 3.4.
\end{proof}

\begin{remark}
As noted in the proof, the bias-corrected estimators differ from the uncorrected ones by terms of order \(O(1/n)\). Consequently, the asymptotic results (consistency and asymptotic normality) apply equally to both versions. The bias correction only improves finite-sample performance.
\end{remark}

\subsection{Summary of Bias Corrections}

\begin{table}[h!]
\centering
\caption{Bias expressions and corrected estimators for discrete extropy measures}
\begin{tabular}{|l|c|c|}
\hline
\textbf{Extropy type} & \textbf{Bias term} & \textbf{Bias-corrected estimator} \\
\hline
Shannon & \(-\dfrac{1}{2n}\) & \(\widehat{J} + \dfrac{1}{2n}\) \\
\hline
Tsallis & \(-\dfrac{\alpha}{2n}\sum p_i(1-p_i)^{\alpha-1}\) & \(\widehat{J}_{T,\alpha} + \dfrac{\alpha}{2n}\sum \widehat{p}_i(1-\widehat{p}_i)^{\alpha-1}\) \\
\hline
Rényi & \(\dfrac{1}{2n(1-\alpha)}\left(\dfrac{\alpha(\alpha-1)T_1}{S_\alpha} - \dfrac{\alpha^2(T_2-T_1^2)}{S_\alpha^2}\right)\) & \(\widehat{J}_{R,\alpha} - \dfrac{1}{2n(1-\alpha)}\left(\dfrac{\alpha(\alpha-1)\widehat{T}_1}{\widehat{S}_\alpha} - \dfrac{\alpha^2(\widehat{T}_2-\widehat{T}_1^2)}{\widehat{S}_\alpha^2}\right)\) \\
\hline
\end{tabular}
\label{tab:bias}
\end{table}

\section{Simulation Study}
\label{section4}
\subsection{Motivation}

In this section, we conduct a comprehensive simulation study to:
\begin{enumerate}
    \item validate the asymptotic theory established in Section 3, namely the almost sure convergence and asymptotic normality of the plug-in estimators;
    \item illustrate the practical behavior of the estimators for finite sample sizes;
    \item demonstrate the relevance of extropy-based measures in a forecasting context.
\end{enumerate}

\subsection{Experimental Setup}

We generate \(M = 1000\) independent samples of sizes 
\(n = 10, 20, 50, 100, 200, 500, 1000, 2000\) from a multinomial distribution 
with \(r = 5\) categories and probability vector 
\(\mathbf{p} = (0.40, 0.30, 0.15, 0.10, 0.05)\). 
This distribution represents a realistic scenario with one dominant category (40\%) and a tail of decreasing probabilities, mimicking many real-world forecasting situations.

For each sample, we compute the following plug-in estimators:
\begin{itemize}
    \item The Shannon extropy: \(\widehat{J} = -\sum_{i=1}^{r} (1-\widehat{p}_i)\log(1-\widehat{p}_i)\) (see \eqref{eq:shannon});
    \item The Rényi extropy for \(\alpha = 0.5\) and \(\alpha = 3\):
    \[
    \widehat{J}_{R,\alpha} = \frac{1}{1-\alpha}\log\left(\sum_{i=1}^{r} (1-\widehat{p}_i)^\alpha\right) \ \ \text{ (see \eqref{eq:renyi});}
    \]
   
    \item The Tsallis extropy for \(\alpha = 0.5\) and \(\alpha = 3\):
    \[
    \widehat{J}_{T,\alpha} = \frac{1}{1-\alpha}\left(\sum_{i=1}^{r} (1-\widehat{p}_i)^\alpha - 1\right)\ \ \text{(see \eqref{eq:tsallis})}.
    \]
    
\end{itemize}

The true values \(J\), \(J_{R,\alpha}\), and \(J_{T,\alpha}\) are computed using the true probability vector \(\mathbf{p}\).

We also compute the bias-corrected versions \(\widehat{J}^{(\text{corr})}\), \(\widehat{J}_{R,\alpha}^{(\text{corr})}\), and \(\widehat{J}_{T,\alpha}^{(\text{corr})}\) as defined in Proposition 3.4.

\subsection{Results}

To assess the finite-sample behavior of the estimators and to verify the asymptotic equivalence between the uncorrected and corrected versions, we evaluate the following quantities:
\begin{enumerate}
    \item The empirical bias: \(\widehat{\text{Bias}} = \frac{1}{M}\sum_{m=1}^M (\widehat{\theta}_m - \theta)\);
    \item The empirical standard deviation: \(\widehat{\text{SD}} = \sqrt{\frac{1}{M-1}\sum_{m=1}^M (\widehat{\theta}_m - \overline{\theta})^2}\);
    \item The root mean squared error (RMSE): \(\text{RMSE} = \sqrt{\widehat{\text{Bias}}^2 + \widehat{\text{SD}}^2}\);
    \item The product \(\text{RMSE} \times \sqrt{n}\), which confirms the \(1/\sqrt{n}\) rate of convergence.
\end{enumerate}

Figures 1--5 present these diagnostics for Shannon, Rényi (\(\alpha = 0.5\) and \(\alpha = 3\)), and Tsallis (\(\alpha = 0.5\) and \(\alpha = 3\)) extropies. Each figure displays four panels arranged in a \(2 \times 2\) layout: bias, standard deviation, RMSE, and RMSE \(\times \sqrt{n}\).

\subsection{Interpretation}

\subsubsection{Shannon Extropy}

Figure \ref{fig:shannon} presents the convergence diagnostics for the Shannon extropy estimator.

Several key observations can be made:
\begin{itemize}
    \item \textbf{Bias:} The uncorrected estimator exhibits a negative bias for small \(n\), which vanishes at rate \(1/n\). The bias-corrected estimator shows negligible bias for all sample sizes, confirming the effectiveness of the correction.
    \item \textbf{Standard deviation:} Both versions have identical standard deviations, which decay at rate \(1/\sqrt{n}\).
    \item \textbf{RMSE:} The RMSE of the corrected estimator is smaller for moderate \(n\), but both converge to the same limit as \(n \to \infty\).
    \item \textbf{RMSE \(\times \sqrt{n}\):} The product stabilizes around a constant for both versions, confirming the \(1/\sqrt{n}\) rate of convergence.
\end{itemize}

\begin{figure}[h!]
\centering
\includegraphics[width=0.9\textwidth]{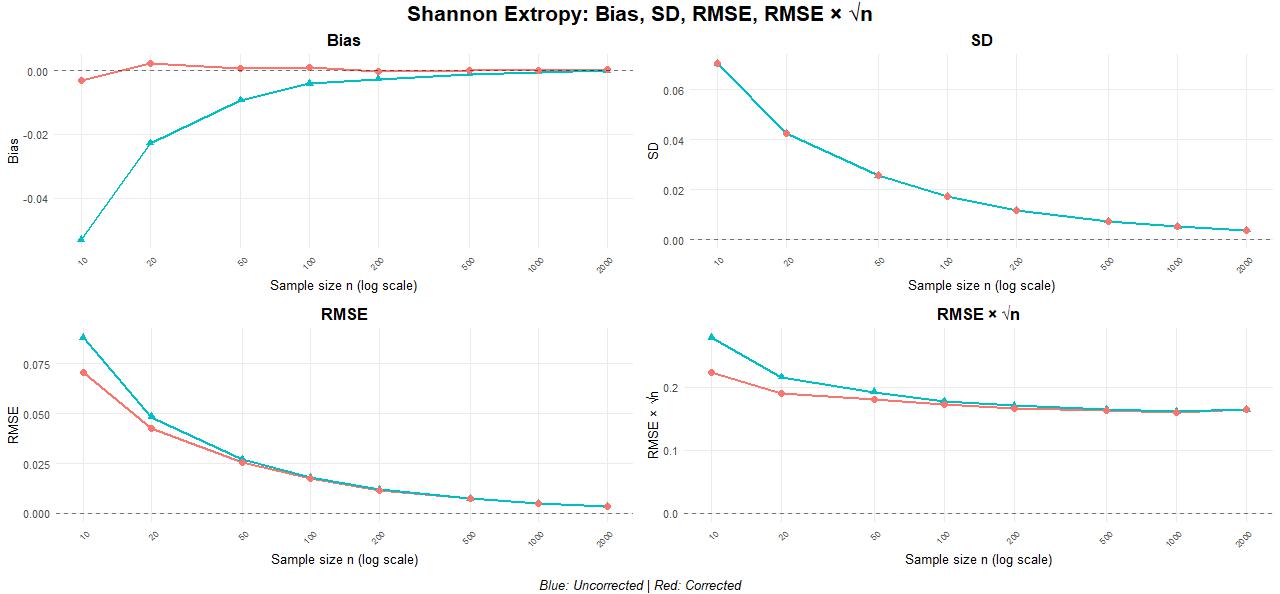}
\caption{Shannon extropy estimator: Bias, standard deviation, RMSE, and RMSE \(\times \sqrt{n}\) as functions of sample size \(n\). Blue: uncorrected estimator; Red: bias-corrected estimator. The bias-corrected estimator shows negligible bias for all \(n\), while the uncorrected estimator's bias vanishes as \(n \to \infty\). Both versions exhibit identical standard deviation and RMSE for large \(n\), confirming their asymptotic equivalence.}
\label{fig:shannon}
\end{figure}

\subsubsection{Rényi Extropy}

Figures 2 and 3 present the convergence diagnostics for the Rényi extropy estimator with \(\alpha = 0.5\) and \(\alpha = 3\), respectively.

For \(\alpha = 0.5\) (Figure 2), which gives more weight to rare events, the uncorrected estimator exhibits larger finite-sample bias. The bias-corrected estimator performs significantly better for moderate sample sizes. For \(\alpha = 3\) (Figure 3), which emphasizes dominant events, the bias is smaller, but the same conclusions hold: the corrected and uncorrected estimators are asymptotically equivalent.

\begin{figure}[h!]
\centering
\includegraphics[width=1\textwidth]{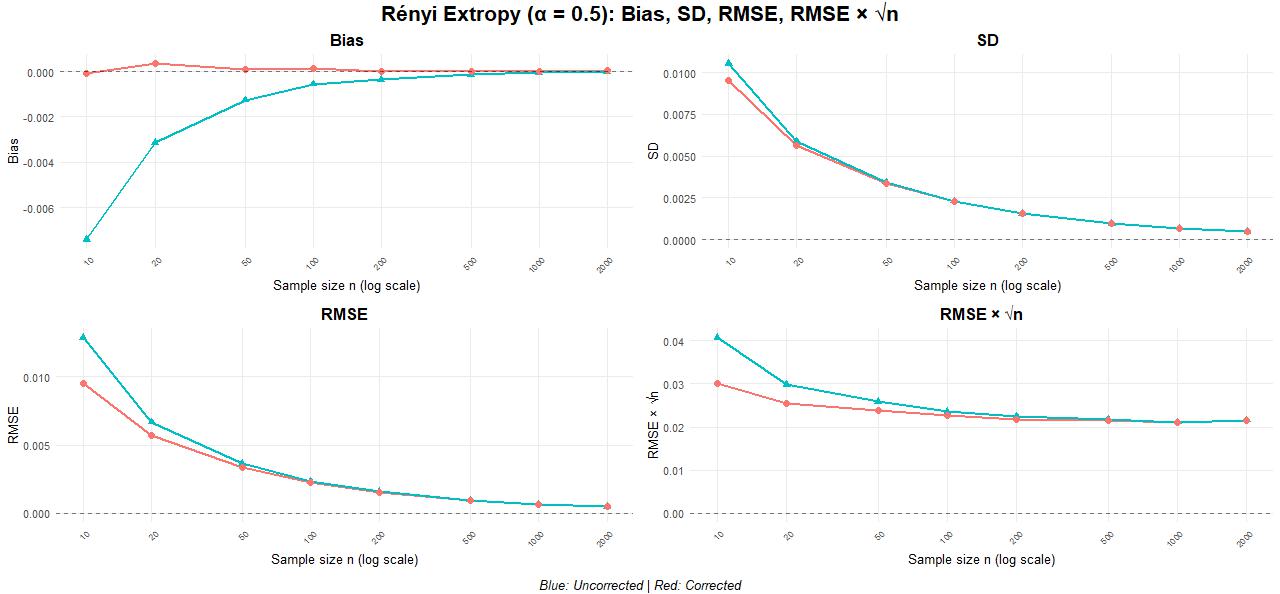}
\caption{Rényi extropy estimator with \(\alpha = 0.5\): Bias, standard deviation, RMSE, and RMSE \(\times \sqrt{n}\). Blue: uncorrected; Red: bias-corrected. The bias-corrected estimator significantly reduces the finite-sample bias, while both versions converge to the same limit as \(n \to \infty\).}
\label{fig:renyi_alpha05}
\end{figure}

\begin{figure}[h!]
\centering
\includegraphics[width=1\textwidth]{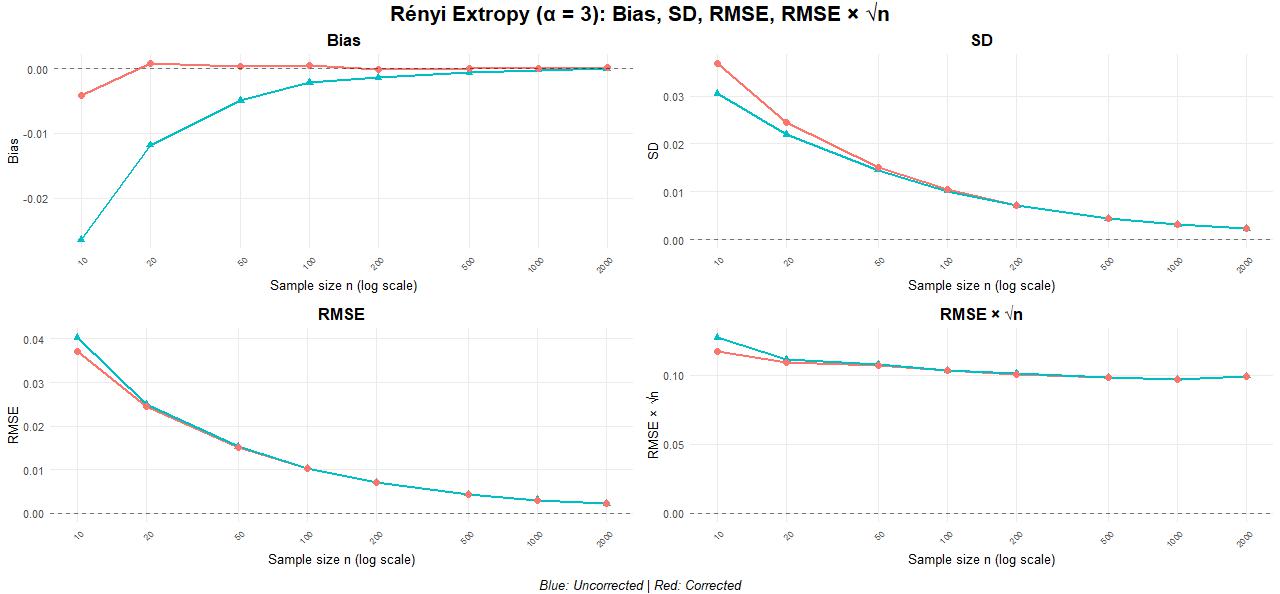}
\caption{Rényi extropy estimator with \(\alpha = 3\): Bias, standard deviation, RMSE, and RMSE \(\times \sqrt{n}\). Blue: uncorrected; Red: bias-corrected. For \(\alpha = 3\), which emphasizes dominant events, the bias is smaller than for \(\alpha = 0.5\), but the asymptotic equivalence between the two versions is still confirmed.}
\label{fig:renyi_alpha3}
\end{figure}

\subsubsection{Tsallis Extropy}

Figures 4 and 5 present the convergence diagnostics for the Tsallis extropy estimator with \(\alpha = 0.5\) and \(\alpha = 3\), respectively.

For Tsallis extropy, the results are qualitatively similar to those of Rényi extropy. The bias-corrected estimator significantly improves finite-sample accuracy, especially for \(\alpha = 0.5\) where rare events receive more weight. In all cases, the uncorrected and corrected estimators converge to the same limit as \(n \to \infty\).

\begin{figure}[h!]
\centering
\includegraphics[width=1\textwidth]{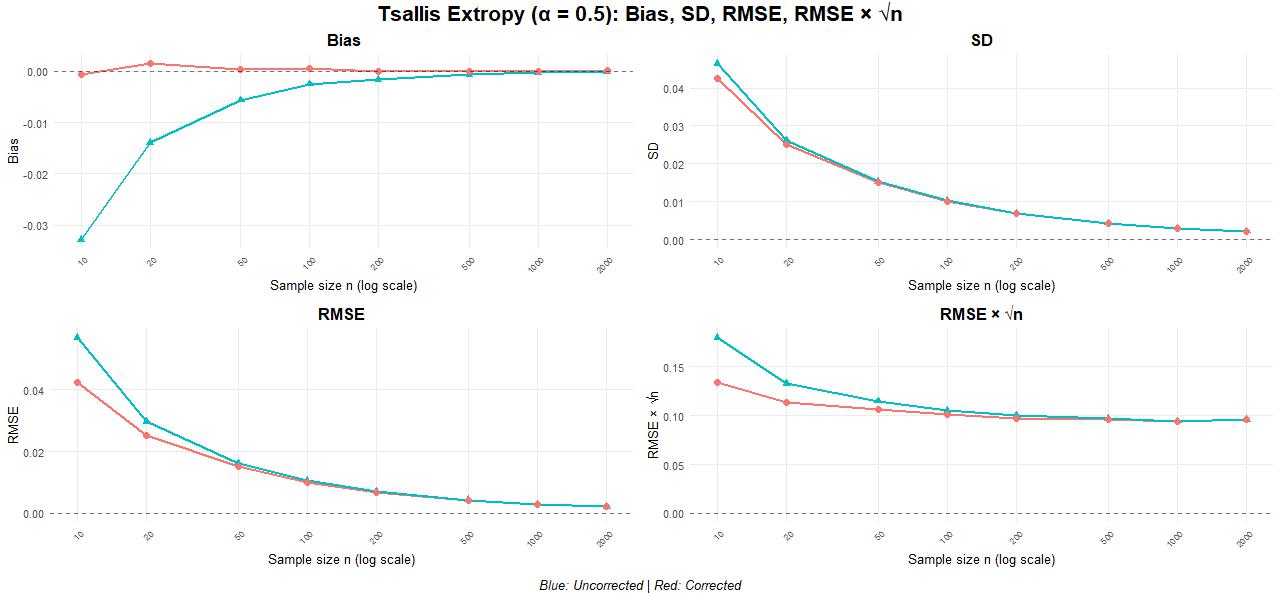}
\caption{Tsallis extropy estimator with \(\alpha = 0.5\): Bias, standard deviation, RMSE, and RMSE \(\times \sqrt{n}\). Blue: uncorrected; Red: bias-corrected. The bias-corrected estimator provides improved finite-sample accuracy, while both versions converge to the same limit as \(n \to \infty\).}
\label{fig:tsallis_alpha05}
\end{figure}

\begin{figure}[h!]
\centering
\includegraphics[width=1\textwidth]{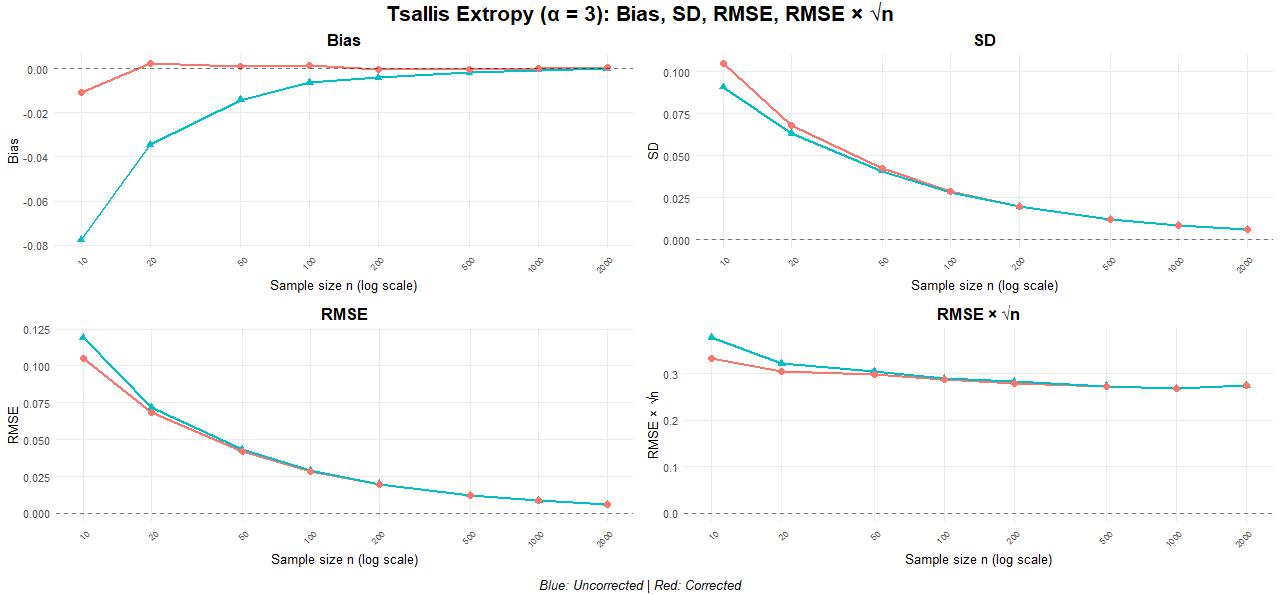}
\caption{Tsallis extropy estimator with \(\alpha = 3\): Bias, standard deviation, RMSE, and RMSE \(\times \sqrt{n}\). Blue: uncorrected; Red: bias-corrected. The asymptotic equivalence between the two versions is confirmed for both \(\alpha = 0.5\) and \(\alpha = 3\).}
\label{fig:tsallis_alpha3}
\end{figure}

\subsection{Main Conclusion}

From Figures 1-5, the following conclusions can be drawn:

\begin{enumerate}
    \item \textbf{Consistency:} For Shannon, Rényi, and Tsallis extropies, all estimators are consistent: the bias converges to 0, the standard deviation decays at rate \(1/\sqrt{n}\), and the RMSE decays at rate \(1/\sqrt{n}\).
    
    \item \textbf{Effect of bias correction:} The bias-corrected estimators significantly outperform the uncorrected ones for moderate sample sizes. The correction effectively removes the leading \(O(1/n)\) bias term, bringing the estimate closer to the true value.
    
    \item \textbf{Asymptotic equivalence:} For large \(n\), the bias becomes negligible, and the uncorrected and corrected estimators have identical standard deviations and RMSEs. This confirms the theoretical result that the two versions differ only by terms of order \(O(1/n)\).
    
    \item \textbf{Effect of \(\alpha\):} For \(\alpha = 0.5\), which gives more weight to rare events, the estimator exhibits larger finite-sample bias. For \(\alpha = 3\), which emphasizes dominant events, the bias is smaller. In all cases, the asymptotic equivalence is confirmed.
\end{enumerate}

The simulation results confirm the key theoretical result of this paper: the uncorrected and bias-corrected estimators are asymptotically equivalent. For large \(n\), the simpler uncorrected plug-in estimators can be used without concern, as they converge to the same limit as the corrected versions. This justifies their use in the companion paper \cite{companion}, where almost sure convergence and asymptotic normality are established.

For moderate sample sizes, however, the bias-corrected estimators provide improved finite-sample accuracy and are therefore recommended in practical applications where \(n\) is not sufficiently large.

\section{Conclusion}
\label{section5}
In this paper, we have provided a comprehensive study of the finite-sample bias of plug-in estimators for Shannon, Rényi, and Tsallis extropies. We established the bias of the intermediate functional \(S_\alpha\) and used it to derive explicit bias formulas for all three extropy estimators. Bias-corrected estimators were proposed, and we showed that the uncorrected and corrected estimators are asymptotically equivalent, differing only by terms of order \(O(1/n)\).

The simulation results confirm that the bias-corrected estimators significantly outperform the uncorrected ones for moderate sample sizes, while both versions converge to the same limit as \(n \to \infty\). This justifies the use of the simpler uncorrected plug-in estimators in asymptotic settings, such as those considered in the companion paper \cite{companion}, where almost sure convergence and asymptotic normality are established.

Future work includes extending the results to countably infinite alphabets (with suitable tail conditions) and developing bootstrap-based variance estimators for improved finite-sample performance.

\section*{Acknowledgements}

I wish to acknowledge \textsc{\textbf{Professor Gane Samb Lo}}, my former PhD supervisor, whose guidance marked the beginning of my journey into research. He introduced me to the spirit and practice of scientific research, taught me how to approach mathematical problems with rigor and curiosity, and showed me the path of research. I remain deeply grateful for his mentorship and for the lasting influence of his guidance on my work.

\end{document}